\documentclass{article}
\usepackage{pgfplots}
\usepackage{fullpage}
\usepackage{amsmath}
\usepackage{amsthm}
\usepackage{amssymb}
\usepackage{url}
\usepackage{graphicx}
\usepackage{verbatim}
\usepackage{url}
\usepackage{graphicx}
\usepackage[caption=false]{subfig}
\usepackage{float}
\usepackage[ruled]{algorithm2e}
\usepackage{tikz}
\usetikzlibrary{calc}
\usetikzlibrary{patterns}
\usepackage{tikz}
\tikzstyle{mybox} = [draw=black, fill=white,  thick,
rectangle, inner sep=10pt, inner ysep=20pt]
\tikzstyle{mybox} = [draw=black, fill=white,  thick,
rectangle, inner sep=2pt, inner ysep=2pt]
\usetikzlibrary{arrows}
\usetikzlibrary{arrows,chains,matrix,positioning,scopes}
\usetikzlibrary{arrows,shapes,positioning}
\usetikzlibrary{decorations.markings}
\tikzstyle arrowstyle=[scale=1]
\tikzstyle directed=[postaction={decorate,decoration={markings,
		mark=at position .65 with {\arrow[arrowstyle]{stealth}}}}]
\tikzstyle reverse directed=[postaction={decorate,decoration={markings,
		mark=at position .65 with {\arrowreversed[arrowstyle]{stealth};}}}]
\newcommand{\boundellipse}[3]
{(#1) ellipse (#2 and #3)
}

\newtheorem{theorem}{Theorem}
\newtheorem{lemma}{Lemma}
\newtheorem{corollary}{Corollary}
\newtheorem{proposition}{Proposition}
\theoremstyle{definition}
\newtheorem{definition}{Definition}
\newtheorem{remark}{Remark}
\newtheorem{example}{Example}

\usepackage{multirow}
\usepackage[utf8]{inputenc}

\begin{document}

\title{Solution of Real Cubic Equations without Cardano's Formula}

\author{Bahman Kalantari\footnote{Emeritus Professor of Computer Science, Rutgers University, Piscataway, New Jersey, U.S.A}}
\date{}
\maketitle

\begin{abstract}
Building on a classification of zeros of cubic equations due to the $12$-th century Persian mathematician Sharaf al-Din Tusi, together with Smale's theory of {\it point estimation}, we derive an efficient recipe for computing high-precision approximation to a real root of an arbitrary real cubic equation. First, via reversible transformations we reduce any real cubic equation into one of four canonical forms with $0$, $\pm 1$ coefficients, except for the constant term as $\pm q$, $q \geq 0$.  Next, given any form, if $\rho_q$ is an approximation to $\sqrt[3]{q}$ to within a relative error of five percent, we prove a {\it seed} $x_0$ in  $\{ \rho_q, \pm .95 \rho_q, -\frac{1}{3}, 1 \}$ can be selected such that in $t$ Newton iterations
$|x_t - \theta_q| \leq \sqrt[3]{q}\cdot 2^{-2^{t}}$ for some real root $\theta_q$. While computing a good seed, even for approximation of $\sqrt[3]{q}$, is considered to be ``somewhat of black art'' (see Wikipedia), as we justify, $\rho_q$ is readily computable from {\it mantissa} and {\it exponent} of $q$. It follows that the above approach gives a simple recipe for numerical approximation of solutions of real cubic equations independent of Cardano's formula.
\end{abstract}



 \section{Introduction.} \label{sec1}

Any cubic equation with real coefficients via affine transformation
can be written in the {\it reduced form} (also called  {\it depressed form}) shown below together with the corresponding {\it Cardano's formula} for a root:
\begin{equation} \label{cardano}
P(x)=x^3+px+q=0, \quad
\theta=\sqrt[3]{-\frac{q}{2} + \sqrt{- \Delta }}  +  \sqrt[3]{-\frac{q}{2} - \sqrt{- \Delta }},
\end{equation}
where $\Delta = -(\frac{q^2}{4} + \frac{p^3}{27})$ is the {\it discriminant}.  Cardano's formula is  credited to several Italians that include,  del Ferro, Tartaglia and Cardano himself who published it in his famous book {\it Ars Magna} in 1545.  For the history behind the formula see  e.g., Irving \cite{Irving} and Katz \cite{Katz}. A surprising by-product of the formula is the emergence of  complex numbers, where $\sqrt{-1}$ was treated as an entity by Bombelli who further analyzed Cardano's book.  In particular, the formula expresses a root as a real number if and only if $\Delta \leq 0$. Otherwise, the three real roots are expressed in terms of complex numbers.  There is a vast literature on how to compute the solutions using Cardano's formula. For example, Zucker \cite{Zucker} suggests an approach to bypass the use of trigonometric functions in DeMoivre's theorem but at the cost of using transcendental functions with transcendental arguments.

While Cardano's formula gives an algebraic expression for the solution of a cubic equation, it does not provide a numerical approximation.  Even if the solution is $\sqrt[3]{q}$ we still need to resort to an iterative method such as Newton's.
Despite the historical significance of the discovery of Cardano's formula and associated mathematical discoveries, it is still valid to ask if in order to approximate a root of a cubic equation it is necessary to use the formula, especially if we can find a direct and efficient algorithm.  For instance, we may consider applying Newton method directly to $P(x)$, where the corresponding {\it iteration function} and {\it orbit} at a {\it seed} $x_0$ are:
\begin{equation}
N_P(x)=x-\frac{P(x)}{P'(x)}, \quad O^+(x_0)=\{x_{t+1}=N_P(x_{t}): t \geq 0\}.
\end{equation}
However, a straightforward application of Newton method may fail. It is well-known that even for a cubic polynomial the orbit of an arbitrary seed may behave chaotically or enter a cycle.  Thus, on the one hand care must be taken to guarantee the orbit converges. On the other hand, one would hope the convergence is fast, taking only a few iterations.  In theory this can be assured if we pick a seed in the {\it quadratic region of convergence} of a simple root of $P(x)$, where roughly speaking the precision doubles with each iteration. However, finding such a point may not be an easy task, even for a cubic polynomial. Smale's {\it approximate zero theory}, also called {\it point estimation} provides a sufficient condition for testing if a given point is in the quadratic region of convergence of a general complex polynomial and the test can be performed efficiently.  Still, computing such a point itself is non-trivial.

In this article we first reduce all real cubic equations into four canonical forms, each defined in terms of a single parameter, $q$. Next, for each form we describe an implicit real interval $I$ of approximate zeros. Then we use the description of $I$ to generate an explicit seed $x_0 \in I$.   Our approach in this reduction builds on a classification of cubic equations due to Sharaf al-Din Tusi, a 12-th century Persian mathematician, recently analyzed in Kalantari and Zaare \cite{KZ}. We will make use of some results from \cite{KZ}, however here we offer new results on the nature of all real cubic equations that in particular turn the problem of numeric approximation of roots of a real cubic equation into a mechanical task, independent of Cardano's formula.

Based on the translations and analysis of the noted mathematics historian of the Golden Age of Islam, Rashed \cite{Rashed1, Rashed2}, Tusi gave a classification of certain cubic equations based on the number of positive zeros  and intervals containing them.  While Tusi correctly identified the intervals containing positive zeros, his technique for deriving them is disputed among math historians, see e.g. Hogendijk \cite{Hogen} and Berggren \cite{Berg}. Regardless, Tusi's work is considered to be deep, in certain respect surpassing the work of Omar Khayyam on such classification. Kalantari and Zaare \cite{KZ} work offers alternative analysis of Tusi's derivation. In particular, they show any real cubic equation, excluding $x^3-q=0$, under affine transformations that may include taking square-root of a coefficient, can be reduced to one of the following two  forms they  call  {\it Tusi form} and  {\it positive normal form}, respectively:
\begin{equation} \label{eqzzzz}
x^3-x^2+\frac{4}{27}\delta =0, \quad \delta \in \mathbb{R}; \quad {\rm ~~and~~}
x^3 + x-q=0, \quad q \in \mathbb{R}.
\end{equation}
In \cite{KZ} we made use of characterization of zeros of Tusi form to give a close connections between the discriminant $\Delta$ in Cardano's formula and the parameter $\delta$ in Tusi form.  Specifically, the reduced polynomial $P(x)$ in (\ref{cardano}) has three real roots if and only if  $\Delta >0$, if and only if
$\sqrt{- \Delta }$ in (\ref{cardano}) is an imaginary  numbers. However, a Tusi form has three real roots if and only if $\delta \in (0,1)$. Moreover,  each of the three intervals $(0,\frac{2}{3})$, $(\frac{2}{3},1)$ and $(-\frac{1}{3}, 0)$ contains a root.

In this article we show that if in addition to the above mentioned transformations
of the variable we also allow inversion, replacing $x$ with $1/x$, then in fact any nontrivial real cubic equations can be reduced into Tusi form in (\ref{eqzzzz}). Thus an efficient procedure for approximation of real roots of a trivial cubic equation as well as a Tusi form, is sufficient for numerical solution of all cubic equations.  With the intension of devising an efficient method for solving cubic equations, first we show
every real cubic equation is reducible to four canonical forms: {\it trivial form},  {\it positive normal form}, {\it Tusi form} with $\delta >1$, and {\it Tusi form} with  $\delta \in [0,1]$. Then for each form we identify an interval of approximate zeros, $I=[a,b]$, containing a real root. Next, from any approximation to $\sqrt[3]{q}$, only to within a relative error of five percent, we derive an explicit approximate zero $x_0 \in I$. As we justify, such approximation to $\sqrt[3]{q}$ is  readily computable from  {\it mantissa} and {\it exponent} of $q$.
Overall, this turns numerical approximation of roots of a real cubic equation into two steps: first reducing the equation into one of the four forms, next starting with an explicit approximate zero, applying a few Newton iterations. In this scheme there is no need to use Cardano's formula which in turn calls for approximation of cube-root of real or complex numbers. Moreover, it is an efficient scheme for computing high-precision approximation to a root.  The complexity bound implies even for $q$ up to  $10^{30}$, essentially at most six Newton iterations suffice.  The approximate solutions is easily convertible back into corresponding solutions for the original cubic equations.

\section{A Classification of Real Cubic Equations.} \label{sec3}

In this section we first define four canonical real cubic  forms. Then in the subsequent section we compute {\it approximate zeros} for each of the four forms and give complexity bound for approximating a real zero of each form.

\begin{definition}  We call the equation,

$x^3-q=0$, $q \geq  0$, {\it trivial form};

$x^3+ x-q=0$, $q \in \mathbb{R}$,  {\it positive normal form};

$x^3- x+q=0$,  $q \in \mathbb{R}$,  {\it negative normal form};

$x^3-x^2+ q=0$, $q \in \mathbb{R}$, {\it Tusi form}, divided into {\it  Type I, II, III}, if $q > \frac{4}{27}$, $q \in [0, \frac{4}{27}]$, and  $q <0$, respectively.
\end{definition}

Proposition \ref{prop1}, Theorem \ref{thm1}  and Theorem \ref{thm2} are from  \cite{KZ}. For the sake of completeness we copy their proofs.   Theorem  \ref{thm1} is a modern version of Tusi's work, where he ambiguously derived the maximum of quadratic function without using derivatives. For different interpretations of Tusi's computation, see \cite{Hogen}, \cite{KZ}, \cite{Rashed1, Rashed2}.

\begin{proposition} \label{prop1}
The reduced forms $x^3+px+q=0$, $p >0$ and $y^3+p'y+q'=0$, $p' <0$ are not reducible to each other under affine transformations. Moreover, a reduced  form $x^3+px+q=0$, with $p \cdot q \not =0$, under the change of variable $x \gets \sqrt{|p|}x$, is reducible to a positive or negative normal form.
\end{proposition}

\begin{proof} The first cubic function is one-to-one while the second is not.  But such property would be preserved under affine transformations.
In the equation $x^3+px+q=0$, $p \cdot q \not =0$, replacing $x$ with $\sqrt{|p|}x$ results in $|p|^{3/2} \big (x^3 + {\rm sign} (p) x \big ) +q=0$.
\end{proof}

\begin{theorem} \label{thm1} {\rm (Tusi)}  The Tusi form  $P(x)=x^3-x^2+\delta \frac{4}{27}=0$ has  three distinct real roots if and only if $\delta \in (0,1)$ and the roots lie in the intervals $(0, \frac{2}{3})$, $( \frac{2}{3},1)$ and $(-\frac{1}{3},0)$. When $\delta =0$ $P(x)$ has a double root at $0$  and when $\delta=1$ a double root at $\frac{2}{3}$.  $P(x)$ has a single real root if and only if $\delta \not \in [0,1]$.
\end{theorem}
\begin{proof} Suppose $\delta \in (0,1)$. $P(x)$ changes sign at $-\frac{1}{3}$, $0$, $\frac{2}{3}$ and $1$. Specifically, $P(-\frac{1}{3})=(-1+\delta)\frac{4}{27} <0$, $P(0) >0$, $P(\frac{2}{3})= (-1+\delta)\frac{4}{27} <0$, and $P(1)>0$. By the {\it intermediate value theorem} there is a root in each respective interval. The cases of $\delta=0$ and $\delta=1$ are trivial.  To show $\delta \not \in [0,1]$ implies a single real zero, it can be shown that the maximum value of $\phi(x)=x^2-x^3$ on $[0, \infty)$ is $\frac{4}{27}$, attained at $\frac{2}{3}$. Moreover, it approaches $\infty$ as $x$ approaches $-\infty$. Also, it is increasing on $[0,\frac{2}{3}]$ and decreasing on $[\frac{2}{3}, \infty)$. Using these it can be shown if $\delta >1$, $P(x)$ has a root in $(-\infty, - \frac{1}{3})$ and  if $\delta <0$ it has a root in $(1, \infty)$.
\end{proof}

\begin{theorem}  \label{thm2} The equation $x^3+px+q=0$, $p <0$ is reducible to Tusi form with $\delta = \frac{1}{2} +  \frac{3\sqrt{3}q}{4 (-p)^{3/2}}$.  There are three distinct real roots if and only if $\Delta >0$.
\end{theorem}
\begin{proof}  Replacing $x$ twice, first by $x- ({-p}/{3})^{1/2}$, subsequently replacing it by $(-3p)^{1/2}x$ results in a Tusi form with $\delta$ as claimed.
From Theorem \ref{thm1} there are three distinct real roots if and only if $\delta \in (0,1)$. Equivalently, $|{3\sqrt{3}q}/{4 (-p)^{3/2}}| < \frac{1}{2}$.
Squaring both sides implies $\Delta >0$. The converse also holds.
\end{proof}

\begin{corollary}  \label{cor1} The negative normal form  $x^3-x+q=0$ corresponds to  Tusi form

(i) with $\delta \in [0,1]$, if and only if $q \in [- \frac{2}{3 \sqrt{3}}, \frac{2}{3 \sqrt{3}}]$;

(ii) with $\delta >1$ if and only if $q > \frac{2}{3 \sqrt{3}}$;

(iii) with $\delta <0$ if and only if $q  < -\frac{2}{3 \sqrt{3}}$.
\end{corollary}
\begin{proof} From Theorem \ref{thm2} the negative normal form  can be written in Tusi form with $\delta = \frac{1}{2} +  \frac{3\sqrt{3}}{4}q$.  Then $\delta \in [0,1]$ if and only if $\frac{3\sqrt{3}}{4}q \in [0, \frac{1}{2}]$, implying (i). Similarly (ii) and (iii) follow.
\end{proof}

\begin{proposition} \label{prop2}
The Type III Tusi form, $x^3-x^2+q=0$, $q <0$,
is reducible to a positive  normal form.
\end{proposition}
\begin{proof} Letting $\widehat q= \sqrt{-q}$, Tusi form becomes $x^3-x^2-\widehat q^2=0$. Replacing $x$ with $-\widehat q/x$, next multiplying by $x^3/\widehat q^2$ we get $x^3+x+\widehat q=0$.
Replacing $x$ with $-x$ gives a positive normal form.
\end{proof}

From Propositions \ref{prop1} and \ref{prop2} we conclude:

\begin{theorem} \label{cor2} Excluding the trivial form, any real cubic equation via transformation of the types $\sqrt{|\alpha|}x$, $\pm (x- \alpha)^{\pm 1}$ is reducible a Tusi form,
$x^3-x^2+q$, $q \in \mathbb{R}$. \qed
\end{theorem}

\begin{remark} As shown in \cite{KZ} under the transformations of the types $\sqrt{|p|}x$ and $\pm (x- \alpha)$ any real cubic equation is reducible, either to a Tusi form or a positive normal form. However, as shown above under the additional transformation of inversion a positive normal form is reducible to a negative normal form and subsequently to a Tusi form.  While this shows the significance of Tusi form, from the point of view of numerical approximation of a real root via Newton method the following theorem summarizes the classification  into four forms and the number of real zeros.
\end{remark}

\begin{theorem} \label{thm3} Any real cubic equation via change of variable of the types $\sqrt{|\alpha|}x$, $\pm (x- \alpha)^{\pm 1}$ is reducible to one of the following four forms:

(1) trivial form (a single real root or a triple root);

(2) positive normal form (a single real root);

(3) Tusi form with $\delta >1$ (a single real root);

(4) Tusi form with $\delta \in [0,1]$ (three real roots, distinct if $\delta \in (0,1)$). \qed

\end{theorem}

\section{Smale's Point Estimation for Cubic Polynomials.}

Smale's approximate zero  theory for a complex polynomial provides a sufficient condition for testing the membership of a point in the quadratic region of  convergence of Newton method. The orbit of such a point rapidly converges to a root of the polynomial.  In this  section we state Smale's condition for a real cubic polynomial $P(x)$. Then in the next four sections we will apply this to the four forms stated in Theorem \ref{thm3}, first by identifying an implicit interval of approximate zeros for each of the forms and subsequently by identifying an explicit  point $x_0$ in each interval.  Additionally, the approximation is easily convertible back into an approximation to a root of the original reduced cubic equation. An example of the latter is given at the end of the section. Note that once we have approximation to one root, by deflation other real or complex roots can be approximated.

\begin{definition}  Let $P(x)$ be a real cubic polynomial. $x_0 \in \mathbb{R}$ is an {\it approximate zero} of  $P(x)$, if  for all $t \geq 1$, $x_{t+1}=N_P(x_t)$  satisfies
$ |x_{t+1}- x_t | \leq 2 \cdot 2^{2^{-t}}|x_1-x_0|$. We let $Z(P)$ denote the set of all approximate zeros.
\end{definition}

\begin{theorem} \label{thm4} {\rm (Smale \cite{Smale86})}  Given a real cubic polynomial $P(x)$, a sufficient condition for $x_0$ to lie in $Z(P)$ is:
\begin{equation} \label{smalecond}
|P(x_0)| \leq \frac{1}{6} \frac{|P'(x_0)|}
{\gamma(x_0)}, \quad {\rm ~~where~~} \gamma(x_0)=\max \left \{  \left |\frac{P''(x_0)}{2 P'(x_0)} \right | , \frac{1}{|P'(x_0)|^{1/2}} \right \}.
\end{equation}
Moreover, for some root $\theta$ of $P(x)$ and for all $t \geq 1$ we have,
\begin{equation} \label{smalebound}
|x_{t} - \theta| \leq 8 \cdot 2^{2^{-t}} |x_0-\theta|.
\end{equation}
We  say $x_0$ is an approximate zero for $\theta$ and write $x_0 \in Z(\theta)$. \qed
\end{theorem}

\begin{remark} Since the Newton orbit of an approximate zero is a Cauchy sequence it must have an accumulation point $\theta$, necessarily a fixed point of $N_P(x)$. But then $\theta$ is a root of $P(x)$. Basically this implies an approximate zero satisfies (\ref{smalebound}). For a formal derivation of (\ref{smalebound})
see Renegar \cite{Renegar}, Proposition 4.1. The constant $1/6$ in (\ref{smalecond}) can actually be replaced with an absolute constant $\alpha_0$, where in Smale's original article for a general polynomial, $\alpha_0 \geq 1/8$. Wang and Zhao \cite{Deren}
show $\alpha_0$ can be taken to be $3-2 \sqrt{2} \geq .17 \geq 1/6$. We have chosen to use $1/6$ only for aesthetic reasons.
\end{remark}

As will be seen, the advantage of reducing a general cubic equation into one of the four canonical forms is that it enables us to compute an explicit approximate zero $x_0$, where its Newton orbit will rapidly converge to a root.  In the next four sections we will consider each of the four forms separately. However, the collective results to be proven can be summarized as the following theorem.

\begin{theorem}  \label{mainthm}  Given $q >0$, let $\rho_q$ be an approximation to $\sqrt[3]{q}$ to within a relative error of five percent, i.e.,
$|\rho_q - \sqrt[3]{q}| \leq .05\sqrt[3]{q}$. An approximate zero $x_0$ and the  location of the corresponding root $\theta_q$ for each polynomial $P(x)$ in the four category of cubic equations: trivial form, positive normal form, type I Tusi form, and type II Tusi form, broken into the two case (4) and (5).

\noindent (1) If $P(x)=x^3-q$, $q \geq 0$, then $x_0= \rho_q$ and $\theta_q=\sqrt[3]{q}$.

\noindent (2) If $P(x)=x^3+x-q$,  $q \geq 3 \sqrt{2}$, then $x_0= .95 \rho_q$ and $\theta_q \in [\sqrt[3]{{2q}/{3}}, \sqrt[3]{q}]$.

\noindent (3) If $P(x)=x^3-x^2+q$, $ q > 8$, then $x_0= - .95 \rho_q$ and $\theta_q \in [- \sqrt[3]{q}, - \sqrt[3]{{2q}/{3}}]$.

\noindent (4) If $P(x)=x^3-x^2+q$,  $q \in [0, \frac{1}{12}]$, then
$x_0 =1$ and $\theta_q \in [\frac{2}{3}, 1]$.

\noindent (5) If $P(x)=x^3-x^2+q$, $q \in [\frac{1}{12}, \frac{4}{27}]$, then  $x_0=-\frac{1}{3}$ and $\theta_q \in [-\frac{1}{3},0]$.

Moreover, for all $t \geq 0$, the Newton iterate $x_t=N_P(x_{t-1})$, satisfies
\begin{equation} \label{errorbd}
|x_t -  \theta_q| \leq
\begin{cases}
\sqrt[3]{q} \cdot 2^{2^{-t}} , & \text{cases (1)-(3)} \\
3 \cdot  2^{2^{-t}}, & \text{cases (4),(5).}
\end{cases}
\end{equation}
\end{theorem}

\begin{remark} \label{remrm}
In the case of (2) for $q \in (0, 3\sqrt{2})$ and the case of (3) for $q  \in (4/27,8]$ we can use (\ref{smalecond}) to find by inspection a small set of approximate zeros. However, we avoid computing such a set for this range of $q$.
\end{remark}

\begin{proposition} \label{prop9}
Let  $q=m \cdot 10^n$, where $m$ is the mantissa and $n$ the exponent of $q$, i.e.
$m \in [1,10)$ is a decimal number and $n \in \mathbb{Z}$. Let $c_n=\sqrt[3]{10^n}$.  Set
$$M=\{(1+.1j)^3: j=0, \dots, 16\}$$
and let $(1+.1j_*)^3$ be the member of $M$ closest to $m$. Then $\rho_m =(1+.1j_*)$
satisfies $|\rho_m -\sqrt[3]{m}| \leq 0.05 \sqrt[3]{m}$ and $\rho_q= c_n \rho_m$ satisfies
$|\rho_q - \sqrt[3]{q}| \leq .05\sqrt[3]{q}$. In particular, in cases of (1)-(3)
in Theorem \ref{mainthm}, the approximate zero $x_0$ can be written in terms of $\rho_m$ and $c_n$.  Moreover, in the case of trivial form $x_t=c_n y_t$, where $y_t = N_Q(y_{t-1})$, $Q(y)=y^3 - m$ and $y_0=\rho_m$ so that to approximate $\sqrt[3]{q}$ it suffices to approximate
$\sqrt[3]{m}$ via Newton method starting with $y_0$, then scale $y_t$  by $c_n$.
\end{proposition}
\begin{proof} Clearly, $1 \leq \sqrt[3]{m} \leq  \sqrt[3]{10} \leq 2.6$.
By the choice of $\rho_m$, $|\rho_m - \sqrt[3]{m}| \leq .05$.  Since  $\sqrt[3]{m} \geq 1$, $|\rho_m - \sqrt[3]{m}| \leq .05 \sqrt[3]{m}$. Multiplying the inequality by $c_n$ implies $|\rho_q - \sqrt[3]{q}| \leq .05 \sqrt[3]{q}$.
Next, using that $\sqrt[3]{q}= c_n \sqrt[3]{m}$ and induction on $t$, it is straightforward to show for all $t \geq 0$,
$x_{t}= N_P(x_{t-1})=c_n N_Q(y_{t-1})= c_n y_{t}$.
\end{proof}

\begin{remark}  We make several observations regarding Theorem \ref{mainthm} and Proposition \ref{prop9}.  Firstly, $\rho_m$ can be computed via binary search, using at most four comparisons between $m$ and numbers in the set $M$.  Secondly, since  $n=3k+r$ for some $k \in \mathbb{Z}$ and $r \in \{0,1,2 \}$, to express $x_0$, it suffices to compute and store a high precision approximation to  $\sqrt[3]{10}$. Thus for each of the five cases in Theorem \ref{mainthm}, $x_0$ can be obtained trivially. Thirdly, in the first three cases the number of Newton iterations to obtain an approximation to $\theta_q$ to a prescribed precision depends only on $q$. In contrast, Cardano's formula is more complicated, possibly involving approximation of cube-root of two complex numbers. Note that the complexity of the proposed algorithm does not depend on the form. The only distinction between the cases is that the orbit of $x_0$ with respect to $x^3-q=0$  is scaled version of the orbit of $y_0$ with respect to  $y^3- m=0$ so that in this case we only need to approximate $\sqrt[3]{m}$ to sufficient precision, then scale it by $\sqrt[3]{10^n}$. Fourthly, from the bound on the error in Theorem \ref{mainthm}, even for $q$ up to $10^{30}$, in $t \leq 6$  Newton iterations $x_t$ approximates a real root to accuracy of at least  $5 \cdot 10^{-10}$.  There is much literature on solving a cubic equation, even for computing cube-root of real numbers. The use of iterative methods such as Newton or higher order methods is inevitable. However, the computation of an appropriate initial seed is often ambiguous. For example, quoting from Wikipedia (\verb|https://en.wikipedia.org/wiki/Cube_root|):

``{\it ... a poor initial approximation of $x_0$ can give very poor algorithm performance, and coming up with a good initial approximation is somewhat of a black art}.''\

\noindent Theorem \ref{mainthm} shows not only approximation of cube-roots can be achieved fast and without any ambiguity in selection of an appropriate initial seed, but solving any real cubic equation approximately is no harder than approximation of a real cube-root in the canonical forms, only to within a relative error of $5$ percent. In the sense of approximation of a root, our approach makes the use of Cardano's formula superfluous.  The classification of cubic equations into the four canonical forms and their properties as stated in Theorem \ref{mainthm} turn the problem of solving a real cubic equation into a trivial mechanical task, where in just a few Newton iterations we can compute very high accuracy approximations.
\end{remark}

\begin{example} Consider the cubic equation,  $F(X)=X^3-15X-4=0$, solved by Rafael Bombelli. Consequently it led to the invention of complex numbers, see e.g. \cite{Irving}.  While the three solutions are real, $4$, $-2 \pm \sqrt{3}$,
Cardano's formula gives the perplexing form $\sqrt[3]{2+ \sqrt{-121}}+ \sqrt[3]{2- \sqrt{-121}}$.
Replacing $X$ with $x- \sqrt{5}$, the new equation becomes $P(x)=x^3-x^2+q=0$, where $q=(10 \sqrt{5}-4)/(3\sqrt{5})^3 \approx 0.06$. Thus $F(x)$ corresponds to a type II Tusi form with three real roots.  If $\theta_X$ is a root of $F(X)$, then $\theta_X=(3 \theta_x-1)\sqrt{5}$, where $\theta_x$ is some root of $P(x)$. Also each Newton iterate $x_t$ with respect to $P(x)$ gives a corresponding iterate $X_t=(3x_t-1) \sqrt{5}$.  Since $P(x)$ is a type II Tusi form with $q \in [0, \frac{21}{48}]$, from Theorem \ref{mainthm}, we set $x_0=1$. The next two Newton iterates are $x_1=.94$ and $x_2=.93$. The corresponding iterates for $F(X)$ are $X_1=4.069$,  $X_2=4.0025$, getting close $4$, a root of $F(x)$.
\end{example}

\section{Solving the Trivial Form.}

\begin{theorem}  \label{thm6} Let $P(x)=x^3-q$, $q \geq 0$.
If $\rho_q$ satisfies
$|\rho_q - \sqrt[3]{q}| \leq .05\sqrt[3]{q}$,
then
$x_0= \rho_q \in Z(\sqrt[3]{q})$. Moreover, $x_t=N_P(x_{t-1})$ satisfies the first bound in (\ref{errorbd}), Theorem \ref{mainthm}. Also, $x_t=c_n y_t$, where $y_t=N_Q(y_{t-1})$, $Q(y)=y^3-m$, $y_0= \rho_m$ as defined in Proposition \ref{prop9}.

\end{theorem}
\begin{proof}
With $a=\sqrt[3]{{2}/{3}}$, $b=\sqrt[3]{2}$, let $I=[a\sqrt[3]{q}, b\sqrt[3]{q}]$. We claim $I \subset Z(P)$. Since $P'(x)=3x^2$, $P''(x)=6x$, $P'''(x)=6$, for $x > 0$,
$\gamma(x) =\max \{ \frac{1}{x}, \frac{1}{\sqrt{3x^2}} \}= \frac{1}{x}$.  Substituting this into Smale's sufficiency condition (\ref{smalecond}), Theorem \ref{thm4}, gives  $|x^3-q| \leq \frac{1}{2}x^3$. Equivalently, $\frac{1}{2}x^3 \leq q \leq  \frac{3}{2}x^3$.
But the set of all $x$ satisfying the latter inequalities is precisely $I$. Since $\sqrt[3]{q}, x_0 \in I$ and $|\rho_q - \sqrt[3]{q}| \leq .05 \sqrt[3]{q}$, $ 8 |x_0 - \sqrt[3]{q}| \leq  \sqrt[3]{q}$. Substituting in Smale's bound (\ref{smalebound}) as applied to $x_0$, implies $x_t$ satisfies the corresponding bound in (\ref{errorbd}), Theorem \ref{mainthm}. The proof of last part is given in Proposition \ref{prop9}.

\end{proof}

\section{Solving the Positive normal Form.}

\begin{lemma}  \label{lem3} Let $P(x)=x^3+x-q=0$, $q >0$. Let $\theta_q$ be its unique real root. Let $a=\sqrt[3]{{2}/{3}}$,
$I=[a\sqrt[3]{q}, \sqrt[3]{q}]$.
If $q \in [0, 3\sqrt{2}]$, then $\theta_q \in (0, 3 \sqrt{2}]$ and if $q  \geq 3 \sqrt{2}$,  $\theta_q \in I.$
\end{lemma}

\begin{proof} $P(0) <0$,  $P(\sqrt[3]{q})= \sqrt[3]{q} >0$ and
$P(a\sqrt[3]{q})= 2q/3 + a\sqrt[3]{q}- q= - q/3 + a\sqrt[3]{q} \leq 0$
if  and only if $\sqrt[3]{q^2}\geq  3 \sqrt[3]{2/3}$, equivalently, $q \geq 3 \sqrt{2}$. By the change of signs in $P(x)$, proof follows.
\end{proof}

\begin{lemma} \label{lem3XX} Let $P(x)=x^3+x-q=0$, $q \geq 3\sqrt{2}$. Suppose $x \geq 1/ \sqrt{6}$. If $x$ satisfies
$|P(x)| \leq \frac{1}{2}x^3$, it is an approximate zero.
\end{lemma}

\begin{proof} Since $P'(x)= 3x^2+1$, $P''(x)= 6x$, it is straightforward to show when $x \geq 1/\sqrt{6}$,  $\gamma(x)=\frac{3x}{3x^2+1}$. We have
$\frac{|p'(x)|}{6\gamma(x)}= \frac{(3x^2+1)^2}{18x}  \geq \frac{9 x^4}{18x}=\frac{1}{2} x^3$.  The proof then follows from (\ref{smalecond}), Theorem \ref{thm4}.
\end{proof}

\begin{theorem}  \label{thm9} Let   $P(x)=x^3+x-q$, $q \geq 3\sqrt{2}$. Let $\theta_q$ be its unique positive root. If $\rho_q$ satisfies
$|\rho_q - \sqrt[3]{q}| \leq .05 \sqrt[3]{q}$, then $x_0=.95 \rho_q \in Z(\theta_q)$. Moreover, $x_t=N_P(x_{t-1})$ satisfies the first bound in (\ref{errorbd}), Theorem \ref{mainthm}.
\end{theorem}

\begin{proof} Let $a=\sqrt[3]{{2}/{3}}$,
$I=[a\sqrt[3]{q}, \sqrt[3]{q}]$. The main part of proof is to show $I \in Z(P)$. By the uniqueness of real root this is equivalent to showing $I \subset Z(\theta_q)$.
From the inequality on $\rho_q$ we get, $.95 \sqrt[3]{q} \leq \rho_q \leq 1.05 \sqrt[3]{q}$. Multiplying this by $.95$ and since $a \leq .95^2$, and $.95 \times 1.05 \leq 1$, we get
\begin{equation} \label{xandq}
a \leq (.95)^2 \leq x_0 \leq .95 \times 1.05 \leq 1.
\end{equation}
Thus $x_0 \in I$. Also, from Lemma \ref{lem3}, $\theta_q \in I$. Thus $ 8|x_0- \theta_q| \leq 8(1-a) \approx 1.01$. However, from (\ref{xandq}) and by considering both cases of $x_0 > \theta_q$ and $x_0 < \theta_q$ we can argue that
$ 8|x_0- \theta_q| \leq  1.0$. Thus if $I \subset Z(P)$, applying Smale's bound (\ref{smalebound}) to $x_0$, implies $x_t$ satisfies the corresponding bound in (\ref{errorbd}), Theorem \ref{mainthm}

Next we  prove $I \subset Z(\theta_q)$.  First we show the endpoints of $I$ are in $Z(\theta_q)$.  Smale's sufficiency condition of Lemma \ref{lem3XX} holds at $ \sqrt[3]{q}$ if $ \sqrt[3]{q} \leq q/2$, equivalently, if $q \geq 2 \sqrt{2}$. At $a \sqrt[3]{q}$  Lemma \ref{lem3XX} holds if $a\sqrt[3]{q} - q/3 \leq 0$, equivalently, if $q \geq 3 \sqrt{2}$. Thus both endpoints of $I$ lie in $Z(\theta_q)$.
 Next we show $I^\circ \subset Z(\theta_q)$. By Lemma \ref{lem3}, $\theta_q \in I$. For $x \in [a\sqrt[3]{q}, \theta_q]$,  $|x^3+x-q|=q-x-x^3$. Also, $q-x-x^3$ is decreasing on this interval, while $x^3/2$ is increasing on $I$.  These together with the fact that Lemma \ref{lem3XX} holds at $x=a\sqrt[3]{q}$  imply the lemma holds for all  $x \in [a\sqrt[3]{q}, \theta_q]$ when $q \geq 3 \sqrt{2}$.  When $q \geq 3 \sqrt{2}$ and $x \in [\theta_q, \sqrt[3]{q}]$, $x^3/2$ is increasing and so is $x^3+x-q$. While at $x=\sqrt[3]{q}$, $x^3/2 \geq x^3+x-q$, we must argue that the inequality holds for all  $x \in [\theta_q, \sqrt[3]{q}]$. To do so, consider
$H(x)=x^3/2- (x^3+x-q)=-x^3/2-x+q$. $H(\sqrt[3]{q}) >0$ and  it goes to $-\infty$ as $x$ approaches infinity. Also, $H(x)$  is reducible to a positive normal form. Thus it has exactly one real zero that must belong to the interval $(\sqrt[3]{q}, \infty)$. This implies when $q \geq 3 \sqrt{2}$,  $x^3/2 \geq x^3+x-q$ for all  $x \in [\theta_q, \sqrt[3]{q}]$. Hence $I \subset Z(P)$ and by uniqueness of the real root, $x_0 \in Z(\theta_q)$.
\end{proof}

\section{Solving  Type I Tusi Form.}

\begin{lemma}  \label{lem1} Let $P(x)=x^3-x^2 +q$, $q >\frac{4}{27}$. A sufficient condition for $x <0$ to be an approximate zero is
$|P(x)| \leq \frac{1}{2} |x^3|$.
\end{lemma}
\begin{proof} $P'(x)= 3x^2-2x$, $P''(x)= 6x-2$. For $x <0$ it is easy to show $\gamma(x) = \frac{|3x-1|}{|3x^2-2x|}$. Since $x <0$, $|3x| < |3x-1|$ and $|3x^2-2x| > |3x^2|$,  $\frac{|P'(x)|}{6\gamma(x)} \geq \frac{1}{2}|x^3|$. Hence the proof of Smale's sufficiency condition, (\ref{smalecond}), Theorem \ref{thm4}.
\end{proof}

\begin{theorem}  \label{thm8} Let $P(x)=x^3-x^2+q=0$, $q \geq 8$.
Let $\theta_q$ be its unique real root. If $\rho_q$ satisfies
$|\rho_q - \sqrt[3]{q}| \leq .05\sqrt[3]{q}$, then $x_0=-.95 \rho_q \in Z(\theta_q)$. Moreover, $x_t=N_P(x_{t-1})$ satisfies the bound (\ref{errorbd}), Theorem \ref{mainthm}.
\end{theorem}
\begin{proof}  Let $a=\sqrt[3]{{2}/{3}}$,
$I=[- \sqrt[3]{q}, -a\sqrt[3]{q}]$.  Clearly, $P(\sqrt[3]{q}) = - q^{2/3} <0$ and $P(-a\sqrt[3]{q})=\frac{1}{3} q - (\frac{2}{3})^2 q^{2/3} \geq 0$ if $q \geq (4/3)^3 \approx 2.37$. This prove $\theta_q \in I$.   As proved in Theorem \ref{thm9}, $.95 \rho_q \in [a\sqrt[3]{q}, \sqrt[3]{q}]$. But this proves $x_0= - .95 \rho \in I$ as defined above.
Next we show $I \subset Z(P)$ and by the uniqueness of the real root this implies $x_0 \in Z(\theta_q)$. Letting $x=-(\alpha q)^{1/3}$ in the inequality $|P(x)| \leq \frac{1}{2} |x^3|$, Lemma \ref{lem1},  and  simplifying, it is equivalent to
\begin{equation} \label{ineqs}
(1- \frac{3}{2}\alpha) \leq \frac{\alpha^{2/3}}{ \sqrt[3]{q}} \leq (1- \frac{1}{2} \alpha).
\end{equation}
It is easy to verify (\ref{ineqs}) is valid for $\alpha \in [\frac{2}{3}, 1]$, $q \geq 8$. Since $x_0, \theta_q \in I$, $|x_0- \theta_q| \leq (1- a) \sqrt[3]{q}$.
Since $8(1-a) \leq 1.02$,  $8|x_0- \theta_q| \leq 1.02 \sqrt[3]{q}$. However, as in the proof of  Theorem \ref{thm9} we can argue $8|x_0- \theta_q| \leq  \sqrt[3]{q}$ so that $x_t$ satisfies the corresponding bound (\ref{errorbd}), Theorem \ref{mainthm}.
\end{proof}

\section{Solving Type II Tusi Form.}

\begin{theorem} \label{thm7} Consider a Tusi form $P(x)=x^3-x^2+ \delta \frac{4}{27}=0$, $\delta \in [0,1]$.

(i) If $\delta \in [0, \frac{27}{48}]$,  $x_0=1 \in Z(P)$ and $O^+(x_0)$ converges to $\theta_\delta  \in (\frac{2}{3},1)$.\

(ii) If  $\delta \in [\frac{21}{48},1]$,  $x_0=- \frac{1}{3} \in Z(P)$ and $O^+(x_0)$ converges  $ \theta_\delta \in (- \frac{1}{3},0)$.\

(iii)  In either case $O^+(x_0)$ satisfies the corresponding bound in  Theorem \ref{mainthm}, see (\ref{errorbd}).
\end{theorem}

\begin{proof} $P'(x)=3x^2-2x$, $P''(x)=6x-2$, $P'''(x)=6$.

(i): We claim when $x  \in (\frac{2}{3}, \infty)$,  $\delta \in [0, \frac{27}{48}]$, then $N_P(x)=(2x^3-x^2 - \delta \frac{4}{27})/(3x^2-2x) > \frac{2}{3}$.
Since $3x^2 - 2x>0$ on $(\frac{2}{3}, \infty)$, the claim is equivalent to showing
$2x^3-3x^2+  \frac{4x}{3}>  \delta \frac{4}{27}$.
The derivative of left-hand-side in the inequality is $6x^2-6x+4/3$. This  is positive on  $(\frac{2}{3}, \infty)$ so that the minimum of $2x^3-3x^2+  \frac{4x}{3}$ on $[\frac{2}{3}, \infty)$ occurs at $\frac{2}{3}$. The minimum is thus $\frac{4}{27}$, greater than $\delta\frac{4}{27}$ for any $\delta \in (0,1)$. It follows that the orbit of $x_0=1$ remains in $(\frac{2}{3}, \infty)$. From Theorem \ref{thm1} (Tusi Theorem), there is a root in $(\frac{2}{3},1)$ and it is the largest root of $P(x)$. Thus if $O^+(x_0)$ converges, it converges to this root.

Next we show $x_0=1 \in Z(\theta_\delta)$ for any $\delta \in [0, \frac{27}{48}]$. Since $P(1)= \delta \frac{4}{27}$, $P'(1)=1$, $P''(1)=4$, from  Theorem \ref{thm4} it
follows that $\gamma(1)=2$. Then (\ref{smalecond}) in Theorem \ref{thm4} reduces to
$\delta \frac{4}{27} \leq \frac{1}{12}$. Equivalently,  $\delta \in [0, \frac{27}{48}]$.

(ii):  Analogous to part (i), we claim for any $x <0 $, $\delta \in [\frac{21}{48},1]$, $N_P(x) <0$. On $(-\infty, 0)$ the denominator of $N_P(x)$ stays positive while its denominator stays negative. Hence $O^+(x_0)$ stays in $(-\infty, 0)$. From Theorem \ref{thm1} there is a root in $(-\frac{1}{3},0)$ and it is the least root. Thus if $O^+(x_0)$ converges, it converges to this root.

Next we show $-\frac{1}{3} \in Z(\theta_\delta)$ for $ \delta \in [0, \frac{21}{48}]$.
We have, $P(-\frac{1}{3})= \frac{4}{27}(-1+\delta)$, $P'(-\frac{1}{3})=1$, $P''(-\frac{1}{3})=4$.  From Theorem \ref{thm4} it follows that $\gamma(\frac{1}{3})=2$. Then (\ref{smalecond}) in Theorem \ref{thm4} reduces to
$|1- \delta| \leq \frac{27}{48}$, implying $\delta \in [\frac{21}{48},1]$.

(iii): Having argued that $O^+(x_0)$ for $x_0=1$ converges to the root in $(\frac{2}{3},1)$ and for $x_0=-\frac{1}{3}$ converges to the root in $(-\frac{1}{3},0)$, it follows that $8|x_0- \theta_\delta | \leq 8 \times \frac{1}{3} \leq 3$.  Substituting  this into  (\ref{errorbd}), Theorem \ref{thm4}, proves the bound on $|x_t - \theta_\delta|$.
\end{proof}

Note that when $\delta \in [\frac{21}{48}, \frac{27}{48}]$, $x_0$ can be taken to be either $-1/3$ or $1$.  We thus have completed the proof of Theorem \ref{mainthm}.

\begin{remark}
Even when the scientific notation of $q$ is not at hand, having selected the approximate zero $x_0$ appropriately  (see Theorem  \ref{mainthm}), for each $t \geq 1$, $x_t$ is a rational function  of $x_0$ and $q$, having guaranteed bound on the absolute error.
\end{remark}

\section*{Concluding Remarks.} In this article we have shown how to approximate a real root of a cubic equation by first reducing it to one of four canonical forms. Except for the trivial form, every real cubic is reducible to Tusi form, $x^3-x^2+q=0$, divided into three types based on $q$. For convenience we chose to convert the Tusi form with negative $q$ into a positive normal form. However, it may be possible to find an approximate zero for this Tusi form as well. Nevertheless, Tusi form and the intervals containing the roots, derived by the Tusi, a $12$-th century mathematician, provided a convenient platform for computing an approximate zero by making use of Smale's theory.  The approximation is convertible back into an approximate root of the original cubic equation. Any round-off error caused by the conversion into a canonical form will likely be corrected by applying a few additional Newton iterations starting with the converted approximation.  The algorithm is direct, without the need to utilize Cardano's formula. In fact Cardano's formula may cause computational issues. For instance, to avoid round-off errors, Fern\'andez Molina et al \cite{Fern}  describe a  method  based on a power
series expansion of Cardano’s formula using Newton’s generalized binomial theorem. They claim  unlike Cardano’s formula and semi-analytical iterative root finders, their method is free from round-off error amplification, for example when the coefficients differ by several orders of magnitude. Their approach still relies on Cardano's formula. While we have no computational experiments with our scheme, it is likely to be practical as it uses Newton iterations, known to be error correcting. Our proposed algorithm is theoretically more efficient than using Cardano's formula which calls for the approximation of two cube-roots of real or complex numbers.  In contrast, our algorithm essentially reduces the problem to that of approximating the cube-root of a single real number, only to within a relative error of  five percent. The rest is computing a few Newton iterations.
The results also offer new insights into the nature of cubic equations. Potential extension include analogous classifications and approximate zeros for solving cubic equations with complex coefficients. The results for cubic equations also raise the question of extensions of classification and approximate zeros to at least quartic and  quintic equations.

\end{document}